\documentclass[11pt]{amsart}
\usepackage{amsmath,amssymb,amsthm,fullpage}
\usepackage{mathrsfs,bbm}

\newtheorem{theorem}{Theorem}
\newtheorem{lemma}[theorem]{Lemma}
\newtheorem{corollary}[theorem]{Corollary}
\newtheorem{prop}[theorem]{Proposition}

\def\N{\mathbb{N}}
\def\S{\mathscr{S}}
\def\E{\mathscr{E}}
\def\Y{\mathsf{Y}}
\def\Exp{\mathbb{E}}
\def\one{\mathbbm{1}}
\begin{document}

\author{Daniel Larsen}
\title{Robust additive bases without minimal subbases}
\email{dlarsen@mit.edu}

\author{Michael Larsen}
\email{mjlarsen@iu.edu}
\thanks{ML was partially supported by NSF grant DMS-2401098 and the Simons Foundation.}

\begin{abstract}
There exists a set $A$ of positive integers such that the number of representations of a large positive integer $m$ as a sum of two elements of $A$ grows 
with a lower bound of order $\log m$, but for which there is no subset $D$ of $A$ minimal for the property that $D+D$ contains all sufficiently large positive integers.
\end{abstract}

\maketitle

\section{Introduction}

Let $\N$ denote the positive integers and $A\subset \N$ be any subset. We write $r_A(m)$ for the number of pairs $(a,b)\in A^2$ such that $a+b=m$ and $a\le b$.
We say such pairs give \emph{representations} of $m$.
We say $A$ is an \emph{additive basis} (of order $2$, always in this paper) if $r_A(m)>0$ for all $m$ sufficiently large.
An additive basis is \emph{minimal} if it contains no proper additive subbasis, i.e. no proper subset which is an additive basis in its own right.
Given $A$, a \emph{summand} of $m$ is an element $a\in A$ such that $m-a\in A$ as well. (A set other than $A$ may be specified.)

Erd\H{o}s and Nathanson \cite{EN79} proved that if $c>\log\bigl(\frac43\bigr)^{-1}$ and $r_A(m) > c\log m$ for all sufficiently large $m$, then $A$ must contain a minimal subbasis.
They conjectured that this is not true for all positive $c$. We prove this conjecture.

\begin{theorem}
\label{main}
There exists $\varepsilon>0$ and a set $A\subset \mathbb N$ such that:
\begin{enumerate}
\item For all sufficiently large $m$, $r_A(m)>\varepsilon\log m$,
\item $A$ contains no minimal additive subbasis of order $2$.
\end{enumerate}
\end{theorem}

Let $X_n=2^{2^n}$, and let
$$I_n=[X_n,X_{n+1})\cap\mathbb N.$$
Henceforth, we will not bother writing $\N$ and will understand intervals to consist only of integers.
We will construct random subsets $A_n\subset I_n$ and then modify $A_n$, 
first by removing a small set of elements to obtain $A'_n$, and then by adding in another small set of  elements to obtain $A''_n$.
Let $A(n)$ (resp. $A''(n)$) be the union of $A_i$ (resp. $A''_i$) as $i$ goes from $1$ to $n$. 
Our sequence of approximations to the final set therefore looks like
$$\cdots \to A''(n-1)\to A''(n-1)\cup A_n\to A''(n-1)\cup A'_n\to A''(n-1)\cup A''_n=A''(n)\to \cdots$$

Let $A''$  be the union of all $A''_i$. 
Once we have determined $A''_n$, we no longer make any changes within $I_n$, so $A''(n)$ is the correct initial part of our final set $A''$.
With probability $1$, the set $A''$ will satisfy the conditions
(1) and (2).  

The idea for constructing $A'_n$ from $A_n$ is first to choose a small subset $B_n\subset I_n$ 
and then to remove from $A_n$  all elements which are summands of elements of $B_n$ with respect to 
$A''(n-1)\cup A_n$, in order to reduce the number of representations of each element of $B_n$ as a sum of two elements of $A''(n-1)\cup A'_n$ to zero. 
We then construct $A''_n$ (which then determines $A''(n)$) by adding some new elements to $A'_n$ in such a way that every element in
$b\in B_n$ has at least $\varepsilon\log b$ representations over $A''(n)$. Furthermore, this modification is carried out in such a way that, for large $n$ and
any subset $D\subset A''(n)$ such that $B_n\subset D+D$, we have almost surely that every element in $I_{n-10}\setminus B_{n-10}$ has at least two different representations as a sum in $D$,
and, moreover, the smallest summand of any element of $B_n$ in $D$ goes to infinity as $n\to \infty$. We will see that this implies that $D$ cannot be a minimal basis.

This strategy is based on \cite{EN89}, which proves that for any fixed $t$, there exists a set $A$ with $r_A(m)\ge t$ for all sufficiently large $m$ which nevertheless does not contain any minimal subbasis. One can think of our sets $B_n$ as being generalizations of their singletons $\{N_n\}$. The key trick of Erd\H{o}s and Nathanson is to arrange matters so that 
$\{N_1,N_2,\ldots\}$ serves as a ``canary in the coal mine,'' whose elements cease to be representable when the set $A$ is thinned, long before elements in its complement.
This makes the analysis of which subsets are subbases much more tractable. The cost of making the $N_n$ ``fragile'' is that $r_A(N_n)$ must be small (in their case bounded).
Because we are able to create many more fragile elements per generation, we are able to spread the load and make the representation count of elements in $B_n$
grow logarithmically (only smaller by a bounded factor than regular elements).

\section{Estimates for $A(n)$}

Let 
$$p(x) = \begin{cases}
\min\biggl(1,40\sqrt{\frac{\log x}x}\biggr) & \text{ if $x\ge 1$}\\
0&\text{ otherwise.}
\end{cases}$$
This is a log-convex function for $x$ sufficiently large.
Let $\Y_2,\Y_3,\Y_4 \ldots$ denote independent Bernoulli random variables such that 
$\Pr[\Y_m=1]=p(m)$
for $m\ge 2$.
We define $A_n = \{m\in I_n\mid \Y_m=1\}$ and $A(n) = A_1\cup\cdots\cup A_n$.
The $p(m)$ are chosen to ensure that
$r_A(m)$ has upper and lower bounds which are constant multiples of $\log m$ for all sufficiently large $m$.

We say a statement about $n$ is \emph{asymptotically true} if with probability $1$ it is true for all but finitely many values $n$.
For instance, we have the following:

\def\ext{\mathrm{ext}}
\begin{lemma}
\label{bounds}
Let $A_n^{\ext} = [X_n/4,X_{n+1})\cap A(n)$. It is asymptotically true that for all $m\in I_n$ we have 
$$r_{A_n^{\ext} }(m) > 160 \log m$$
and
$$r_{A(n)}(m) \ll \log m.$$
\end{lemma}

To be clear, the second statement means that there exists an absolute constant $c$ such that with probability $1$, for all $n$ sufficiently large, for all $m\in I_n$ we have $r_{A(n)}(m) < c \log X_n$.

\begin{proof}
We have
$$r_{A_n^{\ext}}(m) = \sum_{i=X_n/4}^{\lfloor m/2\rfloor} \Y_i \Y_{m-i}.$$
The random variables $\Y_i \Y_{m-i}$ for $i$ from $X_n/4$ to $\lfloor m/2\rfloor$ are independent and Bernoulli with expectation $p(i)p(m-i)$ (unless $i=m/2$, in which case it is just $p(m/2)$.)
As $\frac{\log x}x$ is log-convex for large $x$, the same thing is true for $p(x)$ for large enough $x$.
For large enough $m$, therefore,
the expectation for this sum is at least
$$\Bigl(\frac m2-\frac{X_n}4-1\Bigr)p(m/2)^2 \ge \Bigl(\frac m4-1\Bigr) p(m/2)^2 > 320\log m.$$
By the lower Chernoff bound for sums of independent Bernoulli variables,
$$\Pr[r_{A_n^{\ext}}(m) \le 160 \log m] \le e^{-40\log m} = m^{-40}.$$
This gives a convergent series in $m$, so the first claim follows from the Borel-Cantelli lemma.

The upper bound follows from the fact
$$\sum_{i=2}^{\lfloor m/2\rfloor} p(i)p(m-i) \le 40^2\sum_{i = 2}^{\lfloor m/2\rfloor} i^{-1/2}(m-i)^{-1/2}\log m \ll (m/2)^{-1/2}\sum_{i = 2}^{\lfloor m/2\rfloor} i^{-1/2}\log m \ll \log m.$$
For some absolute constant $c>6$, this is less than $(c/2) \log m$. The second claim in the lemma now follows by using the upper Chernoff bound 

$$\Pr[r_{A_n}(m)\ge c\log m]  \le e^{-\frac {c\log m}6} = m^{-c/6},$$
which is convergent.
\end{proof}

\begin{corollary}
\label{members}
It is asymptotically true that $|A_n| \ll X_n \sqrt{\log X_n}.$
\end{corollary}

We will need bounds for the number of solutions to certain systems of  linear equations in $A(n)$.
Specifically, we want to estimate the probability that there exist distinct triples $(x_i,y_i,z_i)\in A(n)$ where $x_i+y_i$
and $y_i+z_i$ are the same for all $i$. 
Our approach is somewhat similar in spirit to that of \cite{RSZ18}.

We fix $q\neq r\in I_n$ and a positive integer $k$ and estimate the expected number of solutions 
in $A(n)^{3k}$ of systems of equations of the following form:
\begin{equation}
\label{system}
x_i+y_i = q,\ y_i+z_i=r,\  1\le i\le k,
\end{equation}
where the triples $(x_i,y_i,z_i)$ are pairwise distinct.

The general features of the problem are already present in the case $k=2$. The expectation can be written
\begin{equation}
\label{example}
\begin{split}
\Exp\Bigl[\sum_{y_1\neq y_2\in [1,X_{n+1})} &\one_{A(n)}(q-y_1)\one_{A(n)}(y_1)\one_{A(n)}(r-y_1)\one_{A(n)}(q-y_2)\one_{A(n)}(y_2)\one_{A(n)}(r-y_2)\Bigr]\\
&\hskip -40pt= \sum_{y_1\neq y_2\in [1,X_{n+1})} \Exp[\one_{A(n)}(q-y_1)\one_{A(n)}(y_1)\one_{A(n)}(r-y_1)\one_{A(n)}(q-y_2)\one_{A(n)}(y_2)\one_{A(n)}(r-y_2)]\\
&\hskip -40pt= \sum_{y_1\neq y_2\in [1,X_{n+1})} \Exp[\Y_{q-y_1}\Y_{y_1}\Y_{r-y_1}\Y_{q-y_2}\Y_{y_2}\Y_{r-y_2}].
\end{split}
\end{equation}

Any term in this sum is associated with a pair $(y_1,y_2) = (b_1,b_2)$ because $(x_1,x_2) = (a_1,a_2)$ and $(z_1,z_2) = (c_1,c_2)$ must then be $(q-b_1,q-b_2)$ and $(r-b_1,r-b_2)$ respectively.
When $a_1,a_2,b_1,b_2,c_1,c_2$ are pairwise distinct, the $\Y$-factors in the summand on the right hand side of \eqref{example}
are independent random variables, so the summand can be written
$$
p(a_1)p(b_1)p(c_1)p(a_2)p(b_2)p(c_2).$$
However, there may be equalities between pairs of $a_1,a_2,b_1,b_2,c_1,c_2$. For instance, if $a_1=c_2$ but all other pairs of terms are distinct, then 
$\Y_{a_1} \Y_{c_2} = \Y_{a_1}$, but otherwise the $\Y$-terms are still independent, so 
the summand in \eqref{example} is
$$p(a_1)p(b_1)p(c_1)p(a_2)p(b_2).$$
If $a_1 = b_2$, then the relation $a_1+a_2 = q = b_1+b_2$ implies that also $a_2=b_1$, and if there are no other relations among the $a_1,a_2,b_1,b_2,c_1,c_2$, the summand is 
$$p(a_1)p(b_1)p(c_1)p(c_2).$$

We can therefore partition the range of summation $(y_1,y_2)\in A(n)^2$ into subsets according to the equivalence relation given by which of $a_1,\ldots,c_2$ are equal to one another.
Each part in the partition is given by some equalities and some inequalities among the pairs of coordinates, but since we are only interested in upper bounds, we can ignore the inequalities. For example,
the sum over the main term in the partition, for which the $a_1,\ldots,c_2$ are pairwise distinct, is bounded above by the sum over all $a_1,\ldots,c_2$.
$$\sum_{y_1,y_2\in [1,X_{n+1})} p(q-y_1)p(y_1)p(r-y_1)p(q-y_2)p(y_2)p(r-y_2).$$
Note that this factors as
\begin{equation}
\label{factor}
\sum_{y_1\in [1,X_{n+1})} p(q-y_1)p(y_1)p(r-y_1) \sum_{y_2\in [1,X_{n+1})} p(q-y_2)p(y_2)p(r-y_2).
\end{equation}

The contribution to the sum \eqref{example} of the terms where $a_1=b_2$ and therefore $a_2=b_1$, but there are no other equalities, is bounded above by
$$\sum_{y_1\in [1,X_{n+1})} p(q-y_1)p(y_1)p(r-y_1)p(r-(q-y_1)).$$

Let  $\Gamma_{q,r}$ denote the infinite dihedral group of linear functions generated by $x\mapsto q-x$ and $x\mapsto r-x$ (which we call the standard generators) under composition.
In general our summands are products of terms of the form $p(\gamma_i(y_i))$ for some variable $y_i$ and some $\gamma_i\in \Gamma_{q,r}$.
We group terms in the same variable and rewrite our sum as a product of sums over individual variables $y_i$ as in \eqref{factor}.

Let $\Gamma^+_{q,r}$ denote the subgroup of $\Gamma_{q,r}$ consisting of translations; it is generated by $x\mapsto x+q-r$.

\begin{lemma}
\label{crude}
Let  $q\neq r\in I_n$, $\delta > 0$, $k\ge 3$, $K\in \N$, and $\gamma_1,\ldots,\gamma_k$ a sequence of elements of $\Gamma_{q,r}$ of 
length $\le K$ in the standard generators. Let $j$ denote the number of indices $i$ such that $\gamma_i\in \Gamma^+_{q,r}$, and assume $1\le j\le k-1$.
If $n$ is large enough in terms of $\delta$, $k$ and $K$, then
$$\sum_{y\in [1,X_{n+1})} p(\gamma_1(y))\cdots p(\gamma_k(y)) \ll X_n^{\delta-\min(j,k-j)/2}.$$
\end{lemma}

\begin{proof}
We order the $\gamma_i$ so that $\gamma_1,\ldots,\gamma_j$ lie in $\Gamma^+_{q,r}$ and the remaining $\gamma_i$ lie in its complement.

Since we are giving up a factor of $X_n^\delta$, we can ignore the $\log$ factors and replace $p(x)$ with the function $x^{-1/2}$ when $x\ge 1$ and $0$ otherwise.
In other words, it suffices to show that
$$\sum_{\substack{y\in [1,X_{n+1})\\ \gamma_i(y)>0, i=1,\ldots,k}} (\gamma_1(y)\cdots \gamma_k(y))^{-1/2} \ll X_n^{\delta-\min(j,k-j)/2}.$$
The range over which the sum is taken is a subinterval of $[1,X_{n+1})$, and translating the coordinate system, we can write the sum as
\begin{equation}
\label{lemma4}
\sum_{y=1}^M \bigl((y+u_1)\cdots(y+u_j)(M+1+v_{j+1}-y)\cdots (M+1+v_{k}-y)\bigr)^{-1/2},
\end{equation}
where 
$$0=u_1 < \cdots < u_j,\ 0=v_{j+1}<\cdots<v_k.$$

For all $1<i\le j$, $u_i-u_{i-1} \ge |q-r|$ and likewise for the $v_i$.
Moreover, all values of $u_i$ and $v_i$ are bounded by $K|q-r|$.
 
We distinguish two cases, depending on the size of $|q-r|$.
First assume $K|q-r| < q/4$. 
In bounding the sum in \eqref{lemma4} over $y\in [1,M/2]$, we may therefore replace each factor $j<i\le k$ by $q^{-1/2}$ up to a constant factor,
so it suffices to bound
$$q^{-(k-j)/2}\sum_{y=1}^M ((y+u_1)\cdots (y+u_j))^{-1/2} \le q^{-(k-j)/2}\sum_{y=1}^M y^{-j/2} \le q^{\delta-\min(j,k-j)/2}.$$
Note that this last inequality is true even in the case $j=1$, as then $k-j-1\ge j$.
%
We handle the sum over $y>M/2$ in the same way, except that we make a change of variables $y\mapsto M+1-y$, and the role of the $u_i$ and $v_i$ 
are interchanged. Now $q\gg X_n$, so our upper bounds imply the lemma in this case.

Finally, we consider the case $K|q-r| \ge q/4$. For any given $y$, only $y+u_1$ and $M+1+v_{j+1}-y$, among the factors in \eqref{lemma4}, can be $o(q)$.
Thus, \eqref{lemma4} can be bounded by
$$\sum_{y=1}^M (y(M+1-y)q^{k-2})^{-1/2} \ll q^{-(k-2)/2}.$$
As $k\ge 3$, we have $k-2\ge \lfloor k/2\rfloor \ge \min(j,k-j)$, so again we get the desired inequality.
\end{proof}

We now consider the case we actually need.

\begin{prop}
\label{systems}
Let $k\ge 17$. It is asymptotically true that for all $q$ and $r$ distinct elements of $I_n$, the system
\eqref{system} has no solution in $A(n)$ with the triples $(x_i,y_i,z_i)$ pairwise distinct.
\end{prop}

\begin{proof}
Fix $q$ and $r$, and consider a solution $(x_i,y_i,z_i)=(a_i,b_i,c_i)$ of \eqref{system} where the triples are pairwise distinct (and therefore the $b_i$ are pairwise distinct as well).
We can express $a_i$ and $c_i$ in terms of $b_i$ as above and
express the probability that for all $i$, $(a_i,b_i,c_i)\in A(n)^3$ as the expectation of 
\begin{equation}
\label{product}
\one_{A(n)}(a_1)\cdots \one_{A(n)}(a_k)\one_{A(n)}(b_1)\cdots \one_{A(n)}(b_k)\one_{A(n)}(c_1)\cdots \one_{A(n)}(c_k).
\end{equation}

Consider, therefore, the equivalence relation $\E$ on $\{1,\ldots,3k\}$ determined by equality of coordinates 
of the vector 
$$(d_1,\ldots,d_{3k}) = (a_1,\ldots,a_k,b_1,\ldots,b_k,c_1,\ldots,c_k).$$
Let $R_{\E}$ denote the set of smallest representatives of equivalence classes in $\E$.
Thus \eqref{product} is equal to
\begin{equation}
\label{shorter product}
\prod_{j\in R_{\E}} \one_{A(n)}(d_j).
\end{equation}

Now, $\E$ determines an equivalence relation $\E'$ on the set $\{1,\ldots,k\}$ as follows. Let $i\sim j$ if some element of
$\{a_i,b_i,c_i\}$ is equivalent to some element of $\{a_j,b_j,c_j\}$ under $\E$, and take the transitive closure.
Define $R_{\E'}$ to be the subset of $[1,k]$ consisting of the smallest representative of each equivalence class of $\E'$.

The linear relations between $x_i$, $y_i$, and $z_i$ imposed by \eqref{system} imply that if $\{a_i,b_i,c_i\}$ meets $\{a_j,b_j,c_j\}$, then $b_j$ can be expressed as $\gamma(b_i)$
for some $\gamma\in \Gamma_{q,r}$.
It follows that if $i$ and $j$ belong to an equivalence class $S$ of $\E'$, then $b_j$ can be expressed as $\gamma(b_i)$ for some $\gamma\in \Gamma_{q,r}$.
If $i$ denotes the element of $R_{\E'}$ representing $S$,
for each $j\in S$, $b_j$ can be expressed as a linear function in $\Gamma_{q,r}$, of length bounded by $|S|$ in terms of the standard generators, evaluated at $b_i$,
so the $a_j$ and the $c_j$ can be expressed in terms of $b_i$ by elements of $\Gamma_{q,r}$ of length $\le |S|+1$.

Taking all of these elements of $\Gamma_{q,r}$ together, at least $|S|/2$ are translations and at least $|S|/2$ are reflections.
Indeed, let $S = S_+\coprod S_-$ where $j\in S$ belongs to $S_+$ if and only if the linear function in $\Gamma_{q,r}$ expressing $b_j$ in terms of $b_i$ is a translation;
thus, $j\in S_-$ means the linear function is a reflection. Since $a_j$ is obtained from $b_j$ by a reflection,
it is expressible in terms of $b_i$ by a translation in $\Gamma_{q,r}$ if and only if $j\in S_-$.
The $b_j$ for $j\in S$ are pairwise distinct, so there are at least $|S_+|$ elements in $\bigcup_{j\in S} \{a_j,b_j,c_j\}$ obtained from $b_i$ by a translation in $\Gamma_{q,r}$
and at least $|S_-|$ by a reflection. On the other hand, the $a_j$ for $j\in S$ are likewise pairwise distinct, so there are at least $|S_-|$ elements in $\bigcup_{j\in S} \{a_j,b_j,c_j\}$ obtained from $b_i$ by a translation and at least $|S_+|$ by a reflection.  So there are at least $\max(|S_+|,|S_-|) \ge |S|/2$ elements  $\bigcup_{j\in S} \{a_j,b_j,c_j\}$
obtained from $b_i$ by translating and at least the same number by reflecting.

%
%
%

Writing $(x_1,\ldots,x_k,y_1,\ldots,y_k,z_1,\ldots,z_k)$ as $(w_1,\ldots,w_{3k})$,
for each $q$ and $r$ and for each partition $\E$, we express each $w_i$ as a linear polynomial $\gamma_{\E,i}(y_{j_{\E,i}})$, 
where $\gamma_{\E,i}\in \Gamma_{q,r}$ and $j_{\E,i}\in R_{\E'}$ depend only on $i$ and $\E$. 
It is not necessarily the case that for each positive integer choices of $(y_j)_{j\in R_{\E'}}$, the 
solution obtained by evaluating the $\gamma_{\E,i}$ actually gives the original partition $\E$; it might give something coarser, but this causes no problems.

Let $J(j)$ denote the set of $i$ such that $j_{\E,i} = j$.
To obtain an upper bound for the contribution of $\E$ to the 
expectation of \eqref{shorter product}, we drop the inequality conditions (since all terms in the sum are non-negative). This gives
\begin{align*}
\sum_{(y_j)_{j\in R_{\E'}}\in [1,X_{n+1})^{|R_{\E'}|}} \Exp\Bigl[\prod_{i\in R_{\E}} \one_{A(n)}(\gamma_{\E,i}(y_{j_{\E,i}}))\Bigr]
&= \prod_{j\in R_{\E'}}\sum_{y_j\in [1,X_{n+1})} \Exp\biggl[\prod_{i\in J(j)}\one_{A(n)}(\gamma_{\E,i}(y_j))\biggr] \\
&=  \prod_{j\in R_{\E'}}\sum_{y_j\in [1,X_{n+1})} \prod_{i\in J(j)\cap R_{\E}}p(\gamma_{\E,i}(y_j)).
\end{align*}
This is a product of $|R_{\E'}|$ terms, each of which is a sum of products of the type bounded in Lemma~\ref{crude}. If $S$ is an equivalence class in $\E'$,
its corresponding sum is therefore bounded above by $X_n^{\delta-|S|/4}$, so the product as a whole is bounded above by $X_n^{k\delta-k/4}$.
Taking $k=17$ and $\delta = 1/(8\cdot 17)$, we get the upper bound $X_n^{-33/8}$. Summing over all possible values 
of $q\neq r\in I_n$, we see that the expected value of the number of solutions of
$$x_1+y_1=\cdots=x_k+y_k\neq y_1+z_1=\cdots =y_k+z_k$$
with the $(x_i,y_i,z_i)$ pairwise distinct
is $\ll X_n^{-1/8}$.  
By Borel-Cantelli, this implies that with probability $1$, this system has no solution in pairwise distinct triples when $n$ is sufficiently large.

\end{proof}

\section{The Main Construction}

Over the course of this construction, we make several assumptions about $n$ which are asymptotically true. 
If, for a particular $n$, any of these assertions fails,
we set $A''_n=A'_n=A_n$.  We begin by assuming $n>10$.

For every $b\in [X_{n+1}/6,X_{n+1}/4)$, let $U_b$ be the set $a\in A_n$ for which $b-a\in  A''(n-1)\cup A_n$.
We choose at random a short sequence $b_1,\ldots,b_{k_n}$ of pairwise distinct terms in $[X_{n+1}/6,X_{n+1}/4)$ where $k_n$ is to be specified shortly.
Let $B_n = \{b_1,\ldots,b_{k_n}\}$. Note that $B_n$ is uniformly distributed over $k_n$-element subsets of the interval.
Let 
$$A'_n = A_n \setminus \bigcup_{b\in B_n} U_b.$$
By construction, $r_{A''(n-1)\cup A'_n}(b) = 0$ for all $b\in B_n$. 
We will show that this process, while eliminating all representations of the elements of $B_n$, hardly affects the number of representations for any other integers.
We begin with a technical lemma.

\begin{lemma}
\label{almost disjoint}
Let $S\subset [X_n/4,X_{n+1})$ be a set of positive integers, and let $k$ and $\ell$ be positive integers. If a subset $Z$ of $[X_{n+1}/6,X_{n+1}/4)$ of
cardinality $k$ is chosen uniformly, 
then the probability that $|S\cap Z| \ge \ell$ is $\ll_\ell \Bigl(\frac{k|S|}{X_n^2}\Bigr)^\ell$.
If instead, $Z\subset \N$ whose elements are each chosen independently, with probability at most $\epsilon$, then the probability that $|S\cap Z| \ge \ell$ is $\le (\epsilon |S|)^\ell$.
\end{lemma}

\begin{proof}
The probability that any particular $\ell$-element subset of $S$ is contained in $Z$ is $\ll_\ell k^\ell X_n^{-2\ell}$ in the first case and
$\le \epsilon^\ell$ in the second case, and there are less than $|S|^\ell$ such $\ell$-element sets.
\end{proof}

\begin{lemma}
\label{not much changed}
If $|A''_{n-i}\setminus A'_{n-i}| \ll X_n^{1/8}$ for $i=1,2,3$ and $k_n \ll X_n^{1/8}$, then it is asymptotically true that for all $m\not\in B_n$,
\begin{equation}
\label{deviation}
0\le r_{A''(n-1)\cup A_n}(m)-r_{A''(n-1)\cup A'_n}(m) \le 37.
\end{equation}

\end{lemma}
\begin{proof}
Let 
$$E = A''(n-1) \setminus A(n-1),$$
so
\begin{equation}
\label{decomp}
A''(n-1)\cup A_n\subset E\cup A(n).
\end{equation}

The quantity in \eqref{deviation} is the number of representations $m=a+b$, with
$$a\in A''(n-1)\cup A_n,\ b\in (A''(n-1)\cup A_n)\setminus (A''(n-1)\cup A'_n) = A_n\setminus A'_n.$$
First consider the number of representations with $a\in E$. 
As $|A''_i| \le |I_i|  < X_{i+1} = X_n^{2^{i+1-n}}$, our hypotheses imply $|E| \ll X_n^{1/8}$. 
It is asymptotically true that for all $m\in I_n$, $|(m-E)\cap A(n)| \le 5$. Indeed,
applying Lemma~\ref{almost disjoint} with $S = m-E$, $Z = A^{\ext}_n$, $\epsilon = p(X_n/4)$, and $\ell=6$, 
the probability that the intersection has $\ge 6$ elements is $\ll \log^3 X_n X_n^{-9/4}$.

It suffices to show that the number of simultaneous solutions to 
\begin{equation}
\label{abc}
a+b=m,\ b+c\in B_n,\ a\in A(n),\ b,c\in A(n)\cup E
\end{equation}
is at most $32$.

For fixed $E$, $A(n)$, and $k_n$, $B_n$ is uniformly distributed among $k_n$ element subsets of $[X_{n+1}/6,X_{n+1}/4)$, so
it is asymptotically true that $B_n$ and  $(A(n)\cup E) + E$ do not meet.  Therefore, we may assume $b,c\in A(n)$. 



By Lemma~\ref{bounds},
$$|A(n)\cap (m-A(n))|\le 2r_{A(n)}(m) \ll \log X_n.$$
By Lemma~\ref{almost disjoint}, therefore, it is asymptotically true that for all $m\in I_n$, 
$$\Bigm|\!\Bigl(\bigl(A(n)\cap (m-A(n))\bigr)+A(n)\Bigr)\cap B_n\!\Bigm|\, \le 2.$$
By Proposition~\ref{systems}, it is asymptotically true that for all $m\in I_n$ and for each $r\in B_n$, the number of representations of $r$ as $b+c\in A(n)$
with $m\in b+A(n)$ is $< 17$, giving a total of at most $2\cdot 16+5=37$ solutions to \eqref{abc}.

\end{proof}

We remark that because $A''(n-1)\subset [1,X_n)$ and $A_n\subset I_n$, these two sets are disjoint, so for any subset $\tilde A$ of $A''(n-1)$, we have
\begin{equation}
\label{variant}
0\le r_{\tilde A\cup A_n}(m)-r_{\tilde A\cup A'_n}(m) \le 37.
\end{equation}

Finally, we add in a small number of elements to guarantee that $r_A(b) > \varepsilon\log b$ for $b\in B_n$, but we do this in a special way.
Let $C_{n-10}$ denote the set of elements $c\in I_{n-10}$ such that 
$$r_{A''(n-10)\cap A^{\ext}_{n-10}}(c) > 100 \log X_{n-10}.$$
For each such $c$, consider all such representations $c = a_i+a'_i$.
Let $\S_c$ denote all sets consisting of one element of $\{a_i,a'_i\}$ for all but one $i\le 60\log_2 X_{n-10}+1$.
Let $k_n$ be the number of pairs $(c,S_c)$ where $c\in C_{n-10}$ and $S_c\in \S_c$. This is now our choice for the size of $B_n$.

Fix a bijection $\pi$ from $\{1,\ldots,k_n\}$ to pairs $(c,S_c)$. By abuse of notation, we write $\pi(b_i)$ for $\pi(i)$. 
If $\pi(b) = (c,S_c)$, we add the set $\{b\}-S_c$ to $A'_n$, doing this for all $b\in B_n$ to obtain
$A''_n$.  

\begin{lemma}
\label{sizes}
We have
$$k_n = O(X_n^{1/8})$$
and
$$|A''_n\setminus A'_n| = O(X_n^{1/4}).$$
\end{lemma}

\begin{proof}
For each $c\in C_{n-10}$ there are at most $2^{60\log_2 X_{n-10}} \log_2 X_{n-10} <X^{121}_{n-10} < X_n^{61/512}$ elements of $\S_c$. Moreover, there are at most $X_{n-10}^2 = X_n^{1/512}$ possible choices of $c$.
This gives the first claim. For each element in $B_n$, we add $|S_c|\le 60\log_2 X_{n-10}$ 
elements to $A'_n$ to get $A''_n$. Therefore, the total number of elements is $\ll  |B_n|\log X_n$.
\end{proof}

\begin{lemma}
\label{BC}
It is asymptotically true that $B_n\cup C_n=I_n$.
\end{lemma}

\begin{proof}
Let $m\in I_n\setminus B_n$. By Lemma~\ref{bounds}, it is asymptotically true that $r_{A_n^{\ext}}(m) > 160\log m$.
Plugging the bounds of Lemma~\ref{sizes} into Lemma~\ref{not much changed}, we may apply \eqref{variant} to $\tilde A = A_n^{\ext} \cap A''(n-1)$ and obtain
\begin{align*}r_{A_n^{\ext}\cap A''(n)}(m) &\ge r_{A_n^{\ext} \cap (A''(n-1)\cup A'_n)}(m) = r_{\tilde A\cup A'_n}(m) \ge r_{\tilde A\cup A_n}(m)-37 \\
&= r_{A_n^{\ext}}(m)-37 > 100 \log m \ge 100 \log X_n
\end{align*}
when $n$ is sufficiently large, implying $m\in C_n$.
\end{proof}

A crucial point is that when $n$ is large, for each $b\in B_n$, the only representations of $b$ as a sum of two elements of $A''(n)$ are those we 
intentionally added in passing from $A'_n$ to $A''_n$.

\begin{lemma}
\label{no interference}
It is asymptotically true that for all $b\in B_n$, if $(c,S_c) = \pi(b)$, then 
$r_{A''(n)}(b) = |S_c|$.
\end{lemma}

\begin{proof}
Let $a''\in A''_n\setminus A'_n$, and suppose some $b'\in B_n$ satisfies
$b'=a+a''$ where $a\in A''_n$.   Then $a''$ is of the form $b-s$, where $s\in S_c$, so $a''\ge X_{n+1}/6 - X_{n-9}$.
If $a\in A''_n\setminus A'_n$, then likewise
$a \ge X_{n+1}/6 - X_{n-9}$, so 
$$b' = a+a'' \ge X_{n+1}/3-2 X_{n-9}>X_{n+1}/4,$$
which is impossible. Therefore, $a\in A'_n\subset A_n$, and the equation $b'-b = a-s$ has a solution with $b,b'\in B_n$, $a\in A_n$, and $s$ in some $S_c$.

We claim that the probability that such an equation has a solution with $b\neq b'$ goes rapidly to $0$ as $n\to \infty$.
Indeed, $B_n$ is independent of $A_n$ and therefore of $A_n-[X_{n-10}/4,X_{n-9})$, which is a set of size $\ll X_n X_n^{1/512}$, while $|B_n| = k_n \ll X_n^{1/8}$.
For any $e\neq0$, the probability that a random subset of $[X_{n+1}/6,X_{n+1}/4-1]$ of cardinality $k_n$ has two elements $b,b'$ with $b-b'=e$ is
$\ll k_n^2/X_{n+1} \ll X_n^{1/4-2}$. By the union bound, the probability of $B-B$ meeting $A_n-[X_{n-10}/4,X_{n-9})$ converges as we sum over $n$.

%
%

\end{proof}


We remark that for $m\in B_n$, 
$$r_{A''(n)}(m) \ge 60 \log_2 X_{n-10} \ge \frac{15}{512\log 2}\log m,$$
and for $m\in C_n$ we have a better constant. This establishes condition (1) of Theorem~\ref{main} for $\varepsilon = \frac{15}{512\log 2} \approx .042$, but we have made no effort to
optimize this constant.

It remains to explain why $A''$, which we will now rename $A$, has no minimal subbasis. 
Let $B$ and $C$ denote the unions of $B_n$ and $C_n$ respectively over all $n\in\N$. 

Suppose that $D$ is any subbasis of $A$. Suppose $c\in C_{n-10}$ for some $n$ satisfies $r_D(c)\le 1$.
Then $A\setminus D$ meets $\{a_i,a'_i\}$ for all but one  $a_i\le a'_i$ for which $a_i+a'_i = c$.
In particular, it contains some element $S_c\in\S_c$ as a subset. Let $b\in B_n$ satisfy $\pi(b) = (c,S_c)$. Then $r_D(b) = 0$. 

Therefore, for all sufficiently large $c\in C$, $r_D(c) \ge 2$. Delete any single element of $D$ to form $D'$. 
Deleting a single element from $D$ can at most reduce $r_D(c)$ by $1$ because $r$ counts representations $a+b$ with $a\le b$.
Then $r_{D'}(c)\ge 1$ for all $c\in C$. We claim $D'$ is itself a basis. It suffices to show that 
every element $a\in A$ is a summand of finitely many elements of the complement of $C$.
By Lemma~\ref{BC}, it is true with probability $1$ that $B\cup C$ is cofinite in $\N$. Thus, every sufficiently large integer $b$ not in $C$
is in $B_n$ for some $n$.  By Lemma~\ref{no interference}, if $b$ is sufficiently large, every summand either lies in either $S_c$ or $b-S_c$ for $\pi(b)=(c,S_c)$.
Since $S_c\subset [X_{n-10}/4,X_{n-9})$, every element of $S_c$ is $\ge X_{n-10}/4$ and every element of $b-S_c\ge X_{n+1}/6-X_{n-9}$.
For fixed $a$, neither of these bounds holds for $n$ sufficiently large.

Thus $D'$ is again a basis, so $D$ is not a minimal basis.

\end{document}